\documentclass[12pt]{article}
\usepackage{amsfonts,amssymb,amsbsy,amsmath,amsthm}
\usepackage[sans]{dsfont}
\usepackage{graphicx}
\usepackage{color}
\usepackage{pst-all}
\usepackage{pstricks}
\usepackage{pst-plot}

\numberwithin{equation}{section}

\usepackage[english]{babel}
\newtheorem{theorem}{Theorem}[section]
\newtheorem{proposition}[theorem]{Proposition}
\newtheorem{lemma}[theorem]{Lemma}
\newtheorem{follow}[theorem]{Corollary}

\newtheorem{pr}[theorem]{Proposition}
\theoremstyle{definition}

\newcommand{\bel}{\begin{equation} \label}
\newcommand{\ee}{\end{equation}}

\newcommand{\one}{\mathds{1}}
\newcommand{\zero}{\mathbb{O}}

\newcommand{\R}{{\mathbb R}}

\newcommand{\re}{{\mathbb R}}

\newcommand{\rd}{{\mathbb R}^{2}}

\begin{document}
\begin{center}{\Large \bf Dirichlet and Neumann Eigenvalues for Half-Plane Magnetic  Hamiltonians}

\medskip

{\sc Vincent Bruneau, Pablo Miranda, Georgi Raikov}

\medskip
\today
\end{center}

\bigskip

{\bf Abstract.} {\small Let $H_{0, D}$ (resp., $H_{0,N}$) be the Schr\"odinger operator in constant magnetic field on the half-plane with Dirichlet (resp., Neumann) boundary conditions, and let $H_\ell : = H_{0, \ell} - V$, $\ell =D,N$, where the scalar potential $V$ is non negative, bounded, does not vanish identically, and decays at infinity.  We compare the distribution of the  eigenvalues of $H_D$ and $H_N$ below the respective infima of the essential spectra. To this end, we construct effective Hamiltonians which govern the asymptotic behaviour of the discrete spectrum of $H_\ell$ near $\inf \sigma_{\rm ess}(H_\ell) = \inf \sigma(H_{0,\ell})$, $\ell = D,N$. Applying these Hamiltonians, we show that $\sigma_{\rm disc}(H_D)$ is infinite even if $V$ has a compact support, while $\sigma_{\rm disc}(H_N)$ could be finite or infinite depending on the decay rate of  $V$. }\\



\section{Introduction}\setcounter{equation}{0}
\label{s1}

Let ${\mathcal O} : = \re_+ \times \re$ with $\re_+ : = (0,\infty)$, $b>0$, and $H_{0,D}$ (resp., $H_{0,N}$)
   be the Friedrichs extension in $L^2(\mathcal{O})$ of the operator
    \bel{dj1}
     -\frac{\partial^2}{\partial x^2} +
    \left(-i\frac{\partial}{\partial y} - bx\right)^2 ,
    \ee
   defined originally on $C_0^\infty ({\mathcal O})$ (resp., on $C_0^\infty (\overline{\mathcal O})$). Thus, $H_{0,D}$ (resp., $H_{0,N}$) is the Dirichlet (resp., Neumann) realization of  \eqref{dj1} which is the half-plane Schr\"odinger operator with (scalar) constant  magnetic field $b$. Hamiltonians of this type arise in various areas of mathematical and theoretical physics: for instance, $H_{0,D}$ and its perturbations are important models in the theory of the quantum Hall effect (see e.g. \cite{bp}), while the spectral properties of $H_{0,N}$ play a central role in the contemporary theory of superconductivity (see \cite{fh}).  It is well known that the spectrum $\sigma(H_{0,\ell})$ of $H_{0,\ell}$, $\ell = D,N$, is purely absolutely continuous, and
$$
\sigma(H_{0,\ell}) = [{\mathcal E}_\ell, \infty)
$$
with ${\mathcal E}_D = b$ (see e.g \cite{bp}) and ${\mathcal E}_N \in (0,b)$ (see e.g \cite{dh}).  Further, assume that $0 \leq V \in L_0^{\infty}({\mathcal O}) : = \left\{u  \in L^{\infty}({\mathcal O})  \, | \, \lim_{|{\bf x}| \to \infty} u({\bf x}) = 0\right\}$.
    Set
  $$
  H_{D} : = H_{0,D} - V, \quad H_{N} : = H_{0,N} - V.
  $$
  By the diamagnetic inequality (see e.g. \cite[Proposition A.2]{hlmw}), the operators
  $V H_{0,\ell}^{-1}$ are compact, and therefore
  $$
   \sigma_{\rm ess}(H_{\ell}) = \sigma_{\rm ess}(H_{0,\ell}) = \sigma(H_{0,\ell}) = [{\mathcal E}_\ell, \infty),  \quad \ell = D,N.
   $$
   However, the interval $(-\infty, {\mathcal E}_\ell)$ may contain discrete eigenvalues of the operator $H_{\ell}$ whose number could be finite or infinite; in the latter case they could accumulate only at ${\mathcal E}_\ell$. The main goal of the present article is to compare qualitatively the behaviour of the discrete spectra of $H_D$ and $H_N$. For $\lambda > 0$ set
   $$
   {\mathcal N}_\ell(\lambda) : = {\rm Tr} \, \one_{(-\infty, {\mathcal E}_\ell - \lambda)}(H_\ell), \quad \ell = D,N;
   $$
   here and in the sequel $\one_{\mathcal I}$ denotes the characteristic function of the set ${\mathcal I}$. Thus ${\mathcal N}_\ell(\lambda)$ is just the number of the eigenvalues of the operator $H_\ell$ lying on the interval $(-\infty, {\mathcal E}_\ell - \lambda)$, and counted with the multiplicities. \\
   In Section \ref{s2} we introduce the effective Hamiltonians which govern the asymptotics of ${\mathcal N}_\ell(\lambda)$, $\ell = D,N$, as $\lambda \downarrow 0$. The effective Hamiltonians in the Dirichlet  and the Neumann case are quite different due to the different nature of the infima ${\mathcal E}_D$ and ${\mathcal E}_N$ of the spectra of $H_{0,D}$ and $H_{0,N}$. Actually, since both operators $H_{0,D}$ and $H_{0,N}$ are analytically fibred over $\re$, their spectra  have a band structure, and the infima of the spectra coincide with the infima of the first (lowest) band functions. In the case of $H_{0,D}$ the infimum of the first band function is its limit at infinity, while in the case of $H_{0,N}$ the corresponding infimum
   is attained at a (unique) $k_* \in \re$. \\
   We believe that the effective Hamiltonians introduced in the present article could be useful also in the spectral analysis of the perturbations of many other analytically fibred operators whose spectral infimum resembles  $\inf{\sigma(H_{0,D})}$ or $\inf{\sigma(H_{0,N})}$.\\
   Applying our  effective Hamiltonians, we obtain several results concerning the asymptotic behaviour of the discrete spectra of $H_D$ and $H_N$. In particular, we show that the Dirichlet Hamiltonian $H_D$ has infinitely many discrete eigenvalues even in the case of compactly supported $V$, while the Neumann operator $H_N$ may have infinitely or finitely many discrete eigenvalues depending on the decay rate of the effective potential $v$ occurring in the corresponding Neumann effective Hamiltonian (see below \eqref{dj12} and \eqref{dj13}); for example, $\sigma_{\rm disc}(H_N)$ is finite if
   $\limsup_{|y| \to \infty} y^2 \sup_{x \in \re_+} V(x,y)$
   is small enough. \\
   The article is organized as follows. In Section \ref{s2} we formulate our main results and briefly comment on them. Section \ref{s3} contains the proofs of the results concerning the Dirichlet Hamiltonian $H_D$, while the proofs of the ones concerning the Neumann Hamiltonian $H_N$ could be found in Section \ref{s4}.
\section{Statement of the main results} \label{s2}
\subsection{Analytic fibering of $H_{0,\ell}$} \label{ss21}

Let ${\mathcal F}$ be the partial Fourier transform with respect to
$y$, i.e.
$$
({\mathcal F}u)(x,k) = (2\pi)^{-1/2} \int_{\re} e^{-iyk} u(x,y)\,dy, \quad (x,k) \in {\mathcal O}.
$$
Then we have
$$
{\mathcal F} H_{0,\ell} {\mathcal F}^* = \int_\R^{\oplus} h_\ell(k) dk, \quad \ell = D, N,
$$
where  $h_D(k)$ (resp., $h_N(k)$)
 is the Friedrichs extension in $L^2(\re_+)$ of the operator
  \bel{dj3}
 - \frac{d^2}{dx^2} + (bx-k)^2, \quad k \in \R,
\ee
 defined originally on $C_0^\infty (\re_+)$ (resp., on $C_0^\infty (\overline{\re_+})$). Thus $h_{D}$ (resp., $h_{N}$) is the Dirichlet (resp., Neumann) realization of  \eqref{dj3}.
Note that the operators $h_\ell$, $\ell = D,N$, are Kato analytic families (see \cite{kato}). Moreover, for each $k \in \re$ the operators $h_\ell(k)$   have  discrete and simple spectra. Let $E_\ell(k)$, $k \in \re$, be the first (lowest) eigenvalue of $h_\ell$, $\ell = D, N$. By the Kato analytic perturbation theory, $E_\ell(k)$ are real analytic functions of $k \in \re$. Evidently,
$$
{\mathcal E}_\ell = \inf_{k \in \re} E_\ell(k).
$$
Let us recall some of the properties of the functions $E_\ell$ which we will need in the sequel. In both cases $\ell = D,N$, we have
$$
\lim_{k \to \infty} E_\ell(k) = b
$$
(see e.g. \cite{bp} for $\ell = D$, and e.g. \cite{dh} for $\ell = N$). Moreover, the mini-max principle easily implies that
$$
E_\ell(k) = k^2 (1 + o(1)), \quad k \to -\infty.
$$
However, $E_D'(k) < 0$ for any $k \in \re$ (see \cite{bp}) while there exists $k_* \in (0,\infty)$ such that $E_N'(k) < 0$ for $k \in (-\infty, k_*)$,
$E_N'(k) > 0$ for $k \in (k_*, \infty)$, $E''_N(k_*) > 0$, and $E_N(k_*) \in (0,b)$ (see \cite{dh}). Thus
$$
{\mathcal E}_D = b = \lim_{k \to \infty} E_D(k), \quad {\mathcal E}_N = E_N(k_*).
$$
Finally, introduce the real valued eigenfunctions $\psi_\ell(\cdot; k)$ satisfying
$$
h_\ell(k) \psi_\ell(\cdot; k) = E_\ell(k) \psi_\ell(\cdot; k), \quad \|\psi_\ell(\cdot;k)\|_{L^2(\re_+)} = 1, \quad k \in \re, \quad \ell = D,N,
$$
such that the mappings $\re \ni k \mapsto \psi_\ell(\cdot; k) \in {\rm Dom}\,(h_\ell)$ are analytic.

\subsection{Main results for Dirichlet Hamiltonians} \label{ss22}

Set
    \bel{61}
\varphi(x) : = \pi^{-1/4} b^{1/4} e^{-bx^2/2}, \quad x \in \re,
    \ee
    \bel{sof9}
\psi_{D,\infty}(x;k) = \varphi(x - k/b),
 \quad x \in \re, \quad k \in \re.
 \ee
Then we have
    \bel{lau40}
-\frac{\partial^2 \psi_{D,\infty}(x;k)}{\partial x^2} + (bx-k)^2
\psi_{D,\infty}(x;k) = b \psi_{D,\infty}(x;k), \quad
\|\psi_{D,\infty}(\cdot;k)\|_{L^2(\re)} = 1.
    \ee
    For $(x,\xi) \in T^*\re$ define the function
    $$
    \Psi_{x,\xi}(k) : = b^{-1/2} e^{-i\xi k} \psi_{D,\infty}(x;k), \quad k \in \re.
    $$
    As explained in \cite[Section 3]{bmr}, the system
    $\left\{\Psi_{x,\xi}\right\}_{(x,\xi) \in T^*\re}$ is overcomplete with respect to the measure $\frac{b}{2\pi} dx d\xi$ (see \cite[Subsection 5.2.3]{bershu} for the general definition of an overcomplete system with respect to a given measure). Introduce the orthogonal projection
    $$
    {\mathcal P}_{x,\xi} : = |\Psi_{x,\xi} \rangle \langle \Psi_{x,\xi}|, \quad (x,\xi) \in T^*\re,
    $$
    acting in $L^2(\re)$, and the anti-Wick-type operator ${\mathcal V} : L^2(\re) \to L^2(\re)$ defined as the weak integral
    $$
    {\mathcal V} : = \frac{b}{2\pi} \int_{{\mathcal O}} V(x,\xi) \, {\mathcal P}_{x,\xi}\, dx d\xi.
    $$
    Since $V \in L_0^{\infty}({\mathcal O})$, the operator ${\mathcal V}$ is compact. Then the effective Hamiltonian which governs the asymptotics of ${\mathcal N}_D(\lambda)$ as $\lambda \downarrow 0$ is $E_D - {\mathcal V} : L^2(\re) \to L^2(\re)$ where $E_D$ should be interpreted here as the multiplier by the function $E_D$. More precisely, we have the following

    \begin{theorem} \label{thdj1}
    Assume $0 \leq V \in L_0^{\infty}({\mathcal O})$. Then for each $\varepsilon \in (0,1)$ we have
    $$
    {\rm Tr}\,\one_{(-\infty, {\mathcal E}_D - \lambda)}(E_D  - (1-\varepsilon){\mathcal V}) + O_\varepsilon (1) \leq
    $$
    $$
    {\mathcal N}_D(\lambda) \leq
    $$
    \bel{dj5}
    {\rm Tr}\,\one_{(-\infty, {\mathcal E}_D - \lambda)}(E_D  - (1+\varepsilon){\mathcal V}) + O_\varepsilon (1), \quad \lambda \downarrow 0.
    \ee
    \end{theorem}
    The proof of Theorem \ref{thdj1} is contained in Subsection \ref{ss31}.\\

    {\em Remarks}:
    (i) Due to the compactness of the operator ${\mathcal V}$ we have $\sigma_{\rm ess}(E_D  + s{\mathcal V}) = \sigma_{\rm ess}(E_D) = [{\mathcal E}_D, \infty)$ for any $s \in \re$ so that ${\rm Tr}\,\one_{(-\infty, {\mathcal E}_D - \lambda)}(E_D  + s{\mathcal V}) < \infty$ for any $s \in \re$ and $\lambda > 0$. \\
    (ii) The operator $E_D + {\mathcal V}$ is quite similar to the effective Hamiltonians which arose in \cite{bmr} where we studied the asymptotic distribution of the discrete spectrum in the gaps of the essential spectrum of the operator $H_0^{\rm Hall} \pm V$ self-adjoint in $L^2(\rd)$
    with
    $$
    H_0^{\rm Hall} : =
     -\frac{\partial^2}{\partial x^2} +
    \left(-i\frac{\partial}{\partial y} - bx\right)^2 + W(x),
    $$
     and a bounded monotone $W$. There are many other similarities of the spectral properties of the perturbations of  $H_0^{\rm Hall}$ and  $H_{0,D}$, but also there exist several essential differences due mainly to the presence of a boundary in the case of $H_{0,D}$. \\

    In Corollary \ref{fdj1} below we will show that even for a non vanishing identically $V$ that has a compact support, ${\mathcal N}_D(\lambda)$ blows up as $\lambda \downarrow 0$, i.e. the operator $H_D$ has infinitely many eigenvalues. In order to formulate this corollary, we need the following notations. Let ${\Omega} \subset \rd$ be a bounded domain. Denote by ${\bf c}_-(\Omega)$ the maximal length of the vertical segments contained in $\overline{\Omega}$. Further, let $B_R(\zeta) \subset {\mathbb C}$ be a disk of radius $R>0$ centered at $\zeta \in {\mathbb C}$. Identifying ${\mathbb C}$ with $\re^2$, set
    $$
    K(\Omega) : = \left\{(\xi, R) \in \re \times \re_+ \, | \,  \; \mbox{there exists} \; \eta \in \re \; \mbox{such that} \; \Omega \subset B_R(\xi + i \eta)\right\},
    $$
    and
    $$
    {\bf c}_+(\Omega) = \inf_{(\xi, R) \in K(\Omega)} R \varkappa\left(\frac{\xi_+}{eR}\right),
    $$
    where $\xi_+ : = \max\{\xi,0\}$,
    $$
    \varkappa(s) : = \left|\left\{ t > 0 \, | \, t \ln{t} < s\right\}\right|, \quad s \in [0, \infty),
    $$
    and $| \cdot |$ denotes the Lebesgue measure.

    \begin{follow} \label{fdj1}
    Assume that $V$ satisfies
    \bel{dj35}
    c_- \one_{\Omega_-}(x,y) \leq V(x,y) \leq c_+ \one_{\Omega_+}(x,y), \quad (x,y) \in {\mathcal O},
    \ee
    where $\Omega_{\pm} \subset {\mathcal O}$ are bounded domains, and $0 < c_- \leq c_+ < \infty$ are some constants. Then we have
    \bel{dj25}
    {\mathcal C}_-  |\ln{\lambda}|^{1/2}(1 + o(1)) \leq {\mathcal N}_D(\lambda) \leq {\mathcal C}_+  |\ln{\lambda}|^{1/2}(1 + o(1)), \quad \lambda \downarrow 0,
    \ee
    with ${\mathcal C}_- : = (2\pi)^{-1} \sqrt{b} {\bf c}_-(\Omega_-)$ and ${\mathcal C}_+ : = e \sqrt{b} {\bf c}_+(\Omega_+)$.
    In particular,
    $$
    \lim_{\lambda \downarrow 0} \frac{\ln{{\mathcal N}_D(\lambda)}}{\ln{|\ln{\lambda}|}} = \frac{1}{2},
    $$
    and, hence, $H_D$ has infinitely many discrete eigenvalues.
    \end{follow}
    {\em Remark}: The constants ${\mathcal C}_{\pm}$ already appeared in \cite[Theorem 6.1]{bmr}. As it is indicated there, we have ${\mathcal C}_- < {\mathcal C}_+$ (in fact, $\frac{{\mathcal C}_+}{{\mathcal C}_-} > e \pi$).

\subsection{Statement of the main results for Neumann Hamiltonians} \label{ss23}
Assume $0 \leq V \in L_0^{\infty}({\mathcal O})$. Introduce the effective potential
    \bel{dj12}
    v(y) : = \int_0^{\infty} V(x,y) \psi_N(x; k_*)^2 dx, \quad y \in \re,
    \ee
    where $\psi_N(\cdot;k)$, $k \in \re$, is the eigenfunction of the operator $h_N(k)$ introduced at the end of Subsection \ref{ss21}.  Set
    \bel{dj75}
    \mu : = \frac{1}{2} E_N''(k_*).
    \ee
    Recall that $\mu > 0$. Then the effective Hamiltonian which governs the asymptotics of ${\mathcal N}_N(\lambda)$ as $\lambda \downarrow 0$ is the self-adjoint operator
    \bel{dj13}
    -\mu \frac{d^2}{dy^2} - v
    \ee
    defined originally on $C_0^{\infty}(\re)$ and then closed in $L^2(\re)$. More precisely, we have the following

    \begin{theorem} \label{thdj2}
    Assume $0 \leq V \in L_0^{\infty}({\mathcal O})$. Then for each $\varepsilon \in (0,1)$ we have
    $$
    {\rm Tr}\,\one_{(-\infty,  - \lambda)}\left(-\mu \frac{d^2}{dy^2}  - (1-\varepsilon)v\right) + O_\varepsilon (1) \leq
    $$
    $$
    {\mathcal N}_N(\lambda) \leq
    $$
    \bel{dj26}
    {\rm Tr}\,\one_{(-\infty,  - \lambda)}\left(-\mu \frac{d^2}{dy^2}  - (1+\varepsilon)v\right) + O_\varepsilon (1), \quad \lambda \downarrow 0.
    \ee
    \end{theorem}
    The  proof of Theorem \ref{thdj2} can be found in Section \ref{s4}.\\

    {\em Remarks}:
    (i) Since $v \in L_0^{\infty}(\re)$, the multiplier by $v$ is an operator relatively compact with respect to $-\mu \frac{d^2}{dy^2}$ in $L^2(\re)$, and  we have $\sigma_{\rm ess}\left(-\mu \frac{d^2}{dy^2}  + sv\right) = [0, \infty)$ for any $s \in \re$, so that ${\rm Tr}\,\one_{(-\infty, - \lambda)}\left(-\mu \frac{d^2}{dy^2}   + sv\right) < \infty$ for any $s \in \re$ and $\lambda > 0$. \\
    (ii) Effective Hamiltonians quite similar to \eqref{dj13} arose in \cite{brs} where, in particular, the asymptotic distribution of the discrete spectrum  of the operator $H_0^{{\rm strip}} - V$ was studied; here the expression for $H_0^{{\rm strip}}$ coincides with \eqref{dj1} but the operator is considered on the strip $(-L,L) \times \re$, with Dirichlet boundary conditions.\\

    In Corollary \ref{fdj2} below we will establish sufficient conditions for the infiniteness and the finiteness of $\sigma_{\rm disc}(H_N)$.
    \begin{follow} \label{fdj2}
   Let $0 \leq V \in L_0^{\infty}({\mathcal O})$. \\
   (i) Assume that there exists $\alpha \in (0,2)$ and constants $\omega_{\pm} \geq 0$ such that
   \bel{dj8}
   \lim_{y \to \pm \infty} |y|^{\alpha} v(y) = \omega_\pm.
   \ee
   Then we have
   \bel{dj10}
   \lim_{\lambda \downarrow 0} \lambda^{\frac{1}{\alpha} - \frac{1}{2}} {\mathcal N}_N(\lambda) = \frac{B\left(\frac{3}{2}, \frac{1}{\alpha} - \frac{1}{2}\right)}{\pi \alpha \sqrt{\mu}} \left(\omega_-^{1/\alpha} + \omega_+^{1/\alpha}\right),
   \ee
   $B$ being the Euler beta function. \\
   (ii) Assume now that \eqref{dj8} holds with $\alpha = 2$. Then we have
   \bel{dj10a}
   \lim_{\lambda \downarrow 0} |\ln{\lambda}|^{-1} {\mathcal N}_N(\lambda) = \frac{1}{2\pi}  \left(\left(\frac{\omega_-}{\mu} - \frac{1}{4}\right)_+^{1/2} + \left(\frac{\omega_+}{\mu} - \frac{1}{4}\right)_+^{1/2}\right).
   \ee
   (iii) Finally, assume  that
   \bel{dj9}
   \limsup_{|y| \to \infty}y^2 v(y) < \frac{\mu}{4}.
   \ee
   Then we have
   \bel{dj10b}
   {\mathcal N}_N(\lambda) = O(1), \quad \lambda \downarrow 0.
   \ee
   \end{follow}
   {\em Remarks}: (i) If at least one of the constants $\omega_{\pm}$ in \eqref{dj8} with $\alpha \in (0,2)$ is  positive, then \eqref{dj10} implies that the operator $H_N$ has infinitely many discrete eigenvalues. Similarly, if at least one of the constants $\omega_{\pm}$ in \eqref{dj8} with $\alpha = 2$ is greater than $\mu/4$, then \eqref{dj10a} again implies that $\sigma_{\rm disc}(H_N)$ is infinite. Finally, \eqref{dj10b} shows that under  assumption \eqref{dj9}, the operator $H_N$ has at most finitely many discrete eigenvalues. Note that the estimate
   $$
   \limsup_{|y| \to \infty}y^2 \sup_{x \in \re_+} V(x,y) < \frac{\mu}{4}
   $$
   evidently implies \eqref{dj9}. \\
   (ii) Relation \eqref{dj10} can be written in a semiclassical form, namely
   $$
    {\mathcal N}_N(\lambda) = \frac{1}{2\pi} \left|\left\{(y,\eta) \in T^*\re \, | \, \mu \eta^2 - v(y) < - \lambda\right\}\right|(1 + o(1)), \quad \lambda \downarrow 0.
    $$
Corollary \ref{fdj2} follows immediately from Theorem \ref{thdj2} and some well known results on the asymptotics of the discrete spectrum of 1D Schr\"odinger operators with decaying potentials which allow us to investigate the behaviour of ${\rm Tr}\,\one_{(-\infty, - \lambda)}\left(-\mu \frac{d^2}{dy^2} -  sv\right)$, $s>0$, as $\lambda \downarrow 0$. In particular,
    the first part of the corollary is quite close to \cite[Theorem XIII.82]{rs4},
the proof of its second part can be easily deduced from  \cite{ks}, while the
third part follows from the Hardy inequality
$$
\int_0^{\infty} |u'(y)|^2 dy \geq \frac{1}{4} \int_0^{\infty} y^{-2} |u(y)|^2 dy, \quad u \in C_0^{\infty}(\re_+),
$$
and the result of \cite[Problem 22, Chapter XIII]{rs4}.

\section{Proofs of the main results for Dirichlet \\ Hamiltonians} \label{s3}
\subsection{Proof of Theorem 2.1}

Since the proof of Theorem \ref{thdj1} is somewhat lengthy, we will divide the exposition into several parts.\\
(i) {\em Auxiliary  results.}
Let $s>0$ and $T=T^*$ be a linear compact operator acting in a
given Hilbert space\footnote{All Hilbert spaces in this article
are supposed to be separable.}. Set
 $$
n_{\pm}(s; T) : = {\rm Tr}\,\one_{(s,\infty)}(\pm T);
$$
thus the functions $n_{\pm}(\cdot; T)$ are respectively the
counting functions of the positive and negative eigenvalues of the
operator $T$. If $T$ is compact but not necessarily
self-adjoint (in particular, $T$ could act between two different
Hilbert spaces), we will use also the notation
$$
n_*(s; T) : = n_+(s^2; T^* T), \quad s> 0;
$$
thus  $n_{*}(\cdot; T)$ is the counting function of the singular
values of $T$. Evidently,
$$
n_*(s; T) = n_*(s; T^*), \quad n_+(s; T^*T) = n_+(s; T T^*), \quad s>0,
$$
and
$$
n_{\pm}(s; T) \leq n_*(s; T), \quad s > 0, \quad T = T^*.
$$
 Let us recall also the well-known Weyl inequalities
    \bel{lau11}
    n_+(s_1 + s_2; T_1 + T_2) \leq n_+(s_1; T_1) + n_+(s_2; T_2)
    \ee
    where $s_j > 0$ and $T_j$, $j=1,2$, are
linear self-adjoint compact operators (see e.g. \cite[Theorem 9.2.9]{birsol}, as well as the Ky Fan inequalities
    \bel{lau13}
    n_*(s_1 + s_2; T_1 + T_2) \leq n_*(s_1; T_1) + n_*(s_2;
T_2), \quad s_1, s_2 > 0,
    \ee
    for compact but not necessarily self-adjoint $T_j$, $j=1,2$, (see e.g. \cite[Subsection 11.1.3]{birsol}).
    Further, let $S_p$, $p \in [1,\infty)$, be the Schatten -- von Neumann class of compact operators, equipped with the norm
    $$
    \| T \|_p : = \left( -\int_0^{\infty} s^p dn_*(s; T) \right)^{1/p}.
    $$
    Then the  Chebyshev-type estimate
    \bel{dj36}
    n_*(s; T) \leq s^{-p} \|T\|_p^p
    \ee
    holds true for any $s > 0$ and $p \in [1, \infty)$.
    Finally, we recall an abstract version of the Birman-Schwinger principle,  suitable for our purposes. Let $T = T^*$ be lower bounded, and ${\mathcal T} : = \inf{\sigma(T)}$. Let $Q \geq 0$ be a self-adjoint operator,  relatively compact in the sense of the quadratic forms with respect to $T$. Then we have
    $$
    {\rm Tr}\,\one_{(-\infty, {\mathcal T} - \lambda)}(T - sQ) =
    $$
    $$
    n_+(s^{-1}; (T-{\mathcal T} + \lambda)^{-1/2} Q (T-{\mathcal T} + \lambda)^{-1/2}) =
    n_*(s^{-1/2} ; Q^{1/2} (T-{\mathcal T} + \lambda)^{-1/2})
    $$
    for any $s>0$ and $\lambda > 0$ (see \cite[Lemma 1.1]{bir}). \\

(ii) {\em An alternative formulation of Theorem \ref{thdj1}}.
    Instead of Theorem \ref{thdj1} we will prove Theorem \ref{thdj3} below which is slightly more general, and more convenient both to prove and apply.
    In order to formulate it, we need the following notations.  For $(x,y) \in \rd$ set
    $$
    V_*(x,y) = \left\{
    \begin{array} {l}
    V(x,y) \quad \mbox{if} \quad x > 0, \\
    0 \quad \mbox{if} \quad x \leq 0.
    \end{array}
    \right.
    $$
    For $\lambda > 0$ and $A \in [-\infty, \infty)$ define $S_D(\lambda; A) : L^2(A, \infty) \to L^2(\rd)$ as the operator with integral kernel
    $$
    (2\pi)^{-1/2} V_*(x,y)^{1/2} e^{iky} \psi_{D,\infty}(x; k) (E_D(k) - {\mathcal E}_D + \lambda)^{-1/2}, \quad k \in (A,\infty), \quad (x,y) \in \rd,
    $$
    the function $\psi_{D,\infty}$ being defined in \eqref{sof9}.
    \begin{theorem} \label{thdj3}
    Let $0 \leq V \in L_0^{\infty}({\mathcal O})$. Then for any $A \in [-\infty, \infty)$ and $\varepsilon \in (0,1)$, we have
    \bel{dj28}
    n_*(1 + \varepsilon; S_D(\lambda; A)) + O_{\varepsilon, A}(1) \leq {\mathcal N}_D(\lambda) \leq   n_*(1 - \varepsilon; S_D(\lambda; A)) + O_{\varepsilon, A}(1), \quad \lambda \downarrow 0.
    \ee
    \end{theorem}
    Applying the Birman-Schwinger principle, we easily find that
    $$
    n_*(s; S_D(\lambda; -\infty)) = {\rm Tr}\,\one_{(-\infty, {\mathcal E}_D - \lambda)}(E_D - s^{-2} {\mathcal V}), \quad s>0,
    $$
    so that \eqref{dj28} with $A = -\infty$ is equivalent to \eqref{dj5}. The proof of Theorem \ref{thdj3} is contained in the remaining several parts of this subsection.\\

    (iii) {\em Spectral localization to a neighbourhood of ${\mathcal E}_D$}. For $k \in \re$ define the orthogonal projection $\pi_D(k) : L^2(\re_+) \to L^2(\re_+)$ by
    \bel{dj37}
    \pi_D(k) : = |\psi_D(\cdot;k)\rangle \langle \psi_D(\cdot;k) |,
    \ee
    the function $\psi_D(\cdot;k)$ being defined at the end of Subsection \ref{ss21}. Using the  decomposition $L^2(\re) = L^2(-\infty,0) \oplus L^2(0,\infty)$, introduce the orthogonal projection $\zero \oplus \pi_D(k)$ where $\zero$ is the zero operator in $L^2(-\infty,0)$. For $\lambda > 0$ and $A \in [-\infty, \infty)$ define the operator
    \bel{sof1}
R_D(\lambda, A):= {\mathcal F}^* \int^\oplus_{(A,\infty)} (E_D(k)-{\mathcal E}_D + \lambda)^{-1/2}(\zero \oplus \pi_D(k))\,dk\,{\mathcal F},
    \ee
    self-adjoint in $L^2(\rd)$.
\begin{pr}\label{prdj1}
Assume that $0 \leq V \in L_0^{\infty}({\mathcal O})$. Then for any $A \in [-\infty, \infty)$ and $\varepsilon \in (0,1)$ we have
\bel{dj29}
     n_*(1; V_*^{1/2} R_D(\lambda; A)) \, \leq \, {\mathcal N}_D(\lambda) \, \leq \,   n_*(1 - \varepsilon; V_*^{1/2} R_D(\lambda; A)) + O_{\varepsilon, A}(1), \quad \lambda \downarrow 0.
    \ee
\end{pr}
\begin{proof}
By the Birman - Schwinger principle,
    \bel{dj30}
    {\mathcal N}_D(\lambda) = n_+(1; V^{1/2} (H_{0,D} - {\mathcal E}_D + \lambda)^{-1} V^{1/2}), \quad \lambda > 0.
    \ee
    Fix $A \in [-\infty, \infty)$ and define the orthogonal projection $P_D(A) : L^2({\mathcal O}) \to L^2({\mathcal O})$ by
    $$
    P_D(A) : = {\mathcal F}^* \int^\oplus_{(A,\infty)}  \pi_D(k) \,dk {\mathcal F}.
    $$
    Then the mini-max principle and the Weyl inequalities imply that the estimates
    $$
    n_+(1; V^{1/2} (H_{0,D} - {\mathcal E}_D + \lambda)^{-1} P_D(A) V^{1/2}) \leq n_+(1; V^{1/2} (H_{0,D} - {\mathcal E}_D + \lambda)^{-1} V^{1/2}) \leq
    $$
    \bel{dj31}
    n_+(1 - \varepsilon; V^{1/2} (H_{0,D} - {\mathcal E}_D + \lambda)^{-1} P_D(A) V^{1/2}) +
    n_+(\varepsilon ; V^{1/2} (H_{0,D} - {\mathcal E}_D + \lambda)^{-1} (I-P_D(A)) V^{1/2})
    \ee
    hold true for any $\lambda > 0$ and $\varepsilon \in (0,1)$. It is well known that the infimum of the second band function of $H_{0,D}$ is equal to $3b > {\mathcal E}_D$ (see e.g. \cite{bp}). Hence, for any $A \in [-\infty, \infty)$ we have
    $$
    \inf{\sigma\left({H_{0,D}}_{|(I-P_D(A)) {\rm Dom}(H_{0,D})}\right)} > {\mathcal E}_D
    $$
    and the compact operator $V^{1/2} (H_{0,D} - {\mathcal E}_D + \lambda)^{-1} (I-P_D(A)) V^{1/2}$ admits a norm limit as $\lambda \downarrow 0$. Therefore,
    \bel{dj31a}
     n_+(\varepsilon; V^{1/2} (H_{0,D} - {\mathcal E}_D + \lambda)^{-1} (I-P_D(A)) V^{1/2}) = O_{\varepsilon, A}(1), \quad \lambda \downarrow 0,
    \ee
    for any $\varepsilon > 0$. Further, for $A \in [-\infty, \infty)$, $\lambda > 0$, introduce the operator
    $$
   {\mathcal R}_D(\lambda, A):= {\mathcal F}^* \int^\oplus_{(A,\infty)} (E_D(k)-{\mathcal E}_D + \lambda)^{-1/2}  \pi_D(k)\,dk\,{\mathcal F},
    $$
    self-adjoint in $L^2({\mathcal O})$. Then we have
     \bel{dj32}
     n_+(s^2; V^{1/2} (H_{0,D} - {\mathcal E}_D + \lambda)^{-1} P_D(A) V^{1/2}) = n_*(s;   {\mathcal R}_D(\lambda; A) V^{1/2}), \quad s>0.
     \ee
     Finally, it is easy to see that the operators $ {\mathcal R}_D(\lambda; A) V^{1/2}$ and $R_D(\lambda; A) V_*^{1/2}$ have the same non zero singular values. Therefore,
     \bel{dj33}
     n_*(s;  {\mathcal R}_D(\lambda, A) V^{1/2}) = n_*(s; R_D(\lambda; A) V_*^{1/2}) = n_*(s; V_*^{1/2} R_D(\lambda; A)), \quad s>0.
     \ee
     Putting together \eqref{dj30} -- \eqref{dj33}, we obtain \eqref{dj29}.
     \end{proof}
     (iv) {\em Asymptotic estimate of $\zero \oplus \pi_D(k)$ as $k \to \infty$.} For $k \in \re$ define the orthogonal projection $\pi_{D,\infty}(k) : L^2(\re) \to L^2(\re)$ by
     \bel{dj38}
     \pi_{D,\infty}(k) : = |\psi_{D, \infty}(\cdot; k)\rangle \langle \psi_{D, \infty}(\cdot; k)|,
     \ee
     the function $\psi_{D, \infty}(\cdot; k)$ being defined in \eqref{sof9}.

     \begin{theorem} \label{thdj4}
     We have
     \bel{dj34}
     \lim_{k \to \infty}\left(E_D(k) - {\mathcal E}_D\right)^{-1/2} \|\pi_{D, \infty}(k) - (\zero \oplus \pi_D(k))\| = 0.
     \ee
     \end{theorem}

Theorem \ref{thdj4} could be regarded as the central ingredient in the proof of Theorem \ref{thdj3}.
We will split its proof  into three  propositions.\\

For $w \in L^2(\re)$ and $k \in \re$ set
$$
(\tau_k w)(x) : = w(x - k/b), \quad x \in \re.
$$
Evidently, $\tau_k$ is a unitary operator in $L^2(\re)$, and its restriction onto $L^2(-k/b, \infty)$ denoted by the same symbol, is a unitary operator between $L^2(-k/b, \infty)$ and $L^2(\re_+)$. Set
$$
\tilde{h}_D(k) : = \tau_k^* h_D(k) \tau_k, \quad k \in \re.
$$
Thus, $\tilde{h}_D(k)$ is the Dirichlet realization of the operator $-\frac{d^2}{dx^2}+b^2x^2$, self-adjoint in $L^2(-k/b, \infty)$. Put
$$
\tilde{\pi}_D(k) : = \tau_k^* \pi_D(k) \tau_k, \quad k \in \re,
$$
the orthogonal projection $\pi_D(k)$ being defined in \eqref{dj37}. Evidently,
$$
\tilde{\pi}_D(k) = |\tilde{\psi}_D(\cdot ; k) \rangle \langle \tilde{\psi}_D(\cdot; k)|
$$
with $\tilde{\psi}_D(\cdot; k) : = (\tau_k^* \psi_D)(\cdot;k)$.
 Further, for $k \in \re$ introduce the self-adjoint operators
$$
h_{\infty}(k) : = -\frac{d^2}{dx^2}+ (bx-k)^2, \quad \tilde{h}_{\infty} : = \tau_k^* h_\infty(k) \tau_k = -\frac{d^2}{dx^2}+ b^2 x^2,
$$
defined originally on $C_0^{\infty}(\re)$ and then closed in $L^2(\re)$. As before, put
$$
\tilde{\pi}_{D,\infty} : = \tau_k^* \pi_{D,\infty}(k) \tau_k,
$$
the orthogonal projection $\pi_{D,\infty}(k)$ being defined in \eqref{dj38}. Then,
    \bel{dj46}
\tilde{\pi}_{D,\infty} = |\tilde{\psi}_{D,\infty} \rangle \langle \tilde{\psi}_{D,\infty}|
    \ee
with $\tilde{\psi}_{D, \infty} : = (\tau_k^* \psi_{D , \infty})(\cdot;k) = \varphi$, the function $\varphi$ being defined in \eqref{61}.
 Therefore, \eqref{dj34}
is equivalent to
    \bel{bx1a}
\lim_{k\to
\infty}(E_D(k)-\mathcal{E}_D)^{-1/2}||\tilde{\pi}_{D,\infty}-(\tilde{\zero}\oplus\tilde{\pi}_D(k))|| = 0,
    \ee
where $\tilde{\zero}$ is the zero operator in $L^2(-\infty,-k/b)$.\\

Define the operator  $\tilde{h}_-(k)$ as the Dirichlet realization of the operator $-\frac{d^2}{dx^2}+b^2x^2$,  self-adjoint in $L^2(-\infty, -k/b)$. Set
     \bel{2}
\Lambda_k:=\tilde{h}_\infty^{-1}-(\tilde{h}_-(k)^{-1}\oplus\tilde{h}(k)^{-1}),
\quad k \in \R.
    \ee
    This is a rank-one operator which acts in
$L^2(\R)$. Let us give an  explicit representation of it. Let $D_\omega$ be  the standard parabolic-cylinder function \cite[Section 19]{abst} which satisfies
the equation
$$
d^2u/dx^2+(\omega+1/2-x^2/4)u=0.
$$
Then the functions
    \bel{theta}
    \Theta(x):=D_{-\frac{1}{2}}(\sqrt{2b}x), \quad x \in \re,
    \ee
    \bel{psi}\Psi(x):=\frac{1}{2\sqrt{b}}D_{-\frac{1}{2}}(-\sqrt{2b}x), \quad x \in \re,
    \ee
satisfy the equation
$$
-\frac{d^2 u}{dx^2}+ b^2 x^2 u =0,
$$
and the limiting relations
$$
\lim_{x\to\infty}\Theta(x)=0, \quad \lim_{x\to-\infty}\Psi(x)=0.
$$
Due to  the choice of the normalization constants in  \eqref{theta} -- \eqref{psi}, the Wronskian of the functions $\Theta$ and $\Psi$ is equal to $1$.
Put \bel{alphak} \alpha_k:=\frac{\Psi(-k/b)}{\Theta(-k/b)}.\ee Then
we have
    \bel{rlambda}
    \Lambda_k=\langle \cdot\,;\rho_k\rangle \rho_k,
    \ee
    with
\begin{equation} \label{sof9a}
\rho_k(x):=\left\{
\begin{array}
{ll}
    \alpha_k^{-1/2}\Psi(x)\,& {\rm if} \quad x\leq-k/b,   \\
\alpha_k^{1/2} \Theta(x)\,& {\rm if} \quad  x\geq-k/b.\\
\end{array}
\right.
\end{equation}

For further references we include here  the following asymptotic formulas
    \bel{asymp}
    \Theta(x) = e^{-bx^2/2}(\sqrt{2b}x)^{-1/2}(1+o(1)), \quad
    \Psi(x)= \frac{1}{\sqrt{2b}} e^{bx^2/2}(\sqrt{2b}x)^{-1/2}(1+o(1)),\quad x \to \infty,
    \ee
(see \cite[Subsection 19.8]{abst}).

\begin{pr}\label{le1} We have
    \bel{sof5}
||\Lambda_k||=\frac{1}{2}k^{-2}(1+o(1)), \quad k \to \infty.
    \ee
    In particular,
    \bel{bx4}
    \lim_{k \to \infty} \|\Lambda_k\| = 0.
    \ee
\end{pr}
\begin{proof}
Due to \eqref{rlambda},
$||\Lambda_k||=||\rho_k||_{L^2(\R)}^2$.  Thus
    \bel{sof7}
||\Lambda_k||=\frac{1}{\alpha_k}\int_{-\infty}^{-k/b}\Psi(x)^2dx+\alpha_k\int_{-k/b}^\infty\Theta(x)^2dx.
    \ee
By \eqref{asymp},
     \bel{asalpha}
\alpha_k=\frac{1}{2\sqrt{2b}}e^{-k^2/b}(1+o(1)),\quad k\to \infty.
    \ee
Therefore we need to  estimate the integrals appearing in \eqref{sof7}.
First,
$$
\int_{-\infty}^{-k/b}\Psi(x)^2dx=\frac{1}{4b}\int_{k/b}^\infty\Theta(x)^2dx=
\frac{1}{4b}\int_{k/b}^\infty e^{-bx^2}(\sqrt{2b}x)^{-1}dx\,(1+o(1)),
\quad k\to \infty.
$$
Integrating by parts, we easily see that
$$
\int_{k/b}^\infty
e^{-bx^2}(\sqrt{2b}x)^{-1}dx=(2b)^{-3/2}e^{-k^2/b}(k/b)^{-2}(1+o(1)),
\quad k\to \infty,
$$
and then
    \bel{l1}\int_{-\infty}^{-k/b}\Psi(x)^2dx=2^{-7/2}b^{-1/2}k^{-2}e^{-k^2/b}(1+o(1)),
    \quad k\to \infty.
    \ee
    In the same way, for the second integral we find that
    \bel{l2}
    \int_{-k/b}^\infty\Theta(x)^2dx = \sqrt{\frac{b}{2}}  e^{k^2/b}k^{-2}(1+o(1)),
\quad k\to \infty.\ee Putting together \eqref{asalpha}, \eqref{l1}, and
\eqref{l2}, we obtain \eqref{sof5}.
\end{proof}

\begin{proposition}\label{pr1}
There exist $K \in \R$ and $C>0$, such that
    \bel{sof8}
||\tilde{\pi}_{D,\infty}-(\tilde{\zero} \oplus\tilde{\pi}_D(k))||\leq C
||\tilde{\pi}_{D,\infty}\Lambda_k||, \quad k > K.
    \ee
\end{proposition}
\begin{proof}
To begin with, let us write $\tilde{\pi}_D(k)$ in a convenient form. If $I$ is the identity operator in $L^2(\R)$, then
$$
\tilde{\pi}_{D,\infty}=\tilde{\pi}_{D,\infty}(\tilde{\zero}\oplus\tilde{\pi}_{D}(k))+\tilde{\pi}_{D,\infty}(I-(\tilde{\zero}\oplus\tilde{\pi}_D(k))) =
$$
$$
\tilde{\pi}_{D,\infty}(\tilde{\zero}\oplus\tilde{\pi}_{D}(k))+\tilde{\pi}_{D,\infty}[\tilde{h}^{-1}_\infty-{\mathcal
E}_D^{-1}-\Lambda_k] [(\tilde{h}_-(k)^{-1}
\oplus\tilde{h}(k)^{-1})-{\mathcal
E}_D^{-1}]^{-1}(I-(\tilde{\zero}\oplus\tilde{\pi}_{D}(k))) =
$$
$$
\tilde{\pi}_{D,\infty}(\tilde{\zero}\oplus\tilde{\pi}_{D}(k))-\tilde{\pi}_{D,\infty}\Lambda_k [(\tilde{h}_-(k)^{-1}-{\mathcal E}_D^{-1})^{-1}
 \oplus(\tilde{h}(k)^{-1}-{\mathcal E}_D^{-1})^{-1}](I-(\tilde{\zero}\oplus\tilde{\pi}_{D}(k))) =
 $$
$$
\tilde{\pi}_{D,\infty}(\tilde{\zero}\oplus\tilde{\pi}_D(k))-\tilde{\pi}_{D,\infty}\Lambda_k\left[
(\tilde{h}_-(k)^{-1}-{\mathcal
E}_D^{-1})^{-1}\oplus(\tilde{h}(k)^{-1}-{\mathcal
E}_D^{-1})^{-1}(I^+-\tilde{\pi}_D(k))\right],
    $$
    where  $I^+$ is the identity operator in  $L^2(-k/b,\infty)$. Similarly,
$$
\tilde{\zero}\oplus\tilde{\pi}_D(k)=\tilde{\pi}_{D,\infty}(\tilde{\zero}\oplus\tilde{\pi}_D(k))-
\left(\tilde{h}_\infty^{-1}-E_D(k)^{-1}\right)^{-1}(I-\tilde{\pi}_{D,\infty})\Lambda_k(\tilde{\zero} \oplus \tilde{\pi}_D(k)).
$$
Hence,
$$
\tilde{\pi}_{D,\infty}-(\tilde{\zero}\oplus\tilde{\pi}_D(k)) =
$$
$$
-\tilde{\pi}_{D,\infty}\Lambda_k\left[
\left(\tilde{h}_-(k)^{-1}-{\mathcal
E}_D^{-1}\right)^{-1}\oplus\left(\tilde{h}(k)^{-1}-{\mathcal
E}_D^{-1}\right)^{-1}(I^+-\tilde{\pi}_D(k))\right] +
$$
$$
\left(\tilde{h}_\infty^{-1}-E_D(k)^{-1}\right)^{-1}(I-\tilde{\pi}_{D,\infty})\Lambda_k(\tilde{\zero}\oplus\tilde{\pi}_D(k)) =
$$
$$
- \left(\tilde{h}_\infty^{-1}-E_D(k)^{-1}\right)^{-1}(I-\tilde{\pi}_{D,\infty})\Lambda_k (\tilde{\pi}_{D,\infty}-(\tilde{\zero}\oplus\tilde{\pi}_D(k)) \, -
$$
$$
\tilde{\pi}_{D,\infty}\Lambda_k\left[
\left(\tilde{h}_-(k)^{-1}-{\mathcal
E}_D^{-1}\right)^{-1}\oplus\left(\tilde{h}(k)^{-1}-{\mathcal
E}_D^{-1}\right)^{-1}(I^+-\tilde{\pi}_D(k) )\right] +
$$
$$
\left(\tilde{h}_\infty^{-1}-E_D(k)^{-1}\right)^{-1}(I-\tilde{\pi}_{D,\infty})\Lambda_k\tilde{\pi}_{D,\infty},
$$
which, combined with Proposition \ref{le1}, implies that for $k$ large enough we have
$$
\tilde{\pi}_{D,\infty}-(\tilde{\zero}\oplus\tilde{\pi}_D(k)) =
$$
$$
-\left(I+
\left(\tilde{h}_\infty^{-1}-E_D(k)^{-1}\right)^{-1}(I-\tilde{\pi}_{D,\infty})\Lambda_k\right)^{-1} \times
$$
$$
\tilde{\pi}_{D,\infty}\Lambda_k\left[
(\tilde{h}_-(k)^{-1}-{\mathcal
E}_D^{-1})^{-1}\oplus(\tilde{h}(k)^{-1}-{\mathcal
E}_D^{-1})^{-1}(I^+-\tilde{\pi}_D(k) )\right] +
$$
$$
\left(I+
\left(\tilde{h}_\infty^{-1}-E_D(k)^{-1}\right)^{-1}(I-\tilde{\pi}_{D,\infty})\Lambda_k\right)^{-1}
\left(\tilde{h}_\infty^{-1}-E_D(k)^{-1}\right)^{-1}(I-\tilde{\pi}_{D,\infty})\Lambda_k\tilde{\pi}_{D,\infty}.
$$

Finally, since the operators
$$
(\tilde{h}_-(k)^{-1}-{\mathcal
E}_D^{-1})^{-1}, \quad (\tilde{h}(k)^{-1}-{\mathcal
E}_D^{-1})^{-1}(I^+-\tilde{\pi}_D(k) ),  \quad
(\tilde{h}_\infty^{-1}-E_D(k)^{-1})^{-1}(I-\tilde{\pi}_{D,\infty}),
$$
are uniformly bounded with respect to $k$ large enough, we obtain \eqref{sof8}.
\end{proof}

\begin{proposition}\label{pr2} We have
    \bel{bx5}
E_D(k)-\mathcal{E}_D={\mathcal E}_D^2||\tilde{\pi}_{D,\infty}\Lambda_k\tilde{\pi}_{D,\infty}||\,(1+o(1)), \quad
k\to\infty.
    \ee
\end{proposition}
\begin{proof}
By the definition of $\Lambda_k$ and the resolvent identity,
    \bel{l3}
    ({\mathcal
E}_D^{-1}-E_D(k)^{-1})\tilde{\pi}_{D,\infty}(\tilde{\zero} \oplus \tilde{\pi}_D(k))
=\tilde{\pi}_{D,\infty}\Lambda_k (\tilde{\zero} \oplus \tilde{\pi}_D(k)).
    \ee
    By Propositions \ref{le1} --  \ref{pr1},
    \bel{sof4}
||\tilde{\pi}_{D,\infty}(\tilde{\zero} \oplus \tilde{\pi}_D(k))
||=1+o(1),\;
||\tilde{\pi}_{D,\infty}\Lambda_k(\tilde{\zero} \oplus \tilde{\pi}_D(k))||=||\tilde{\pi}_{D,\infty}\Lambda_k\tilde{\pi}_{D,\infty}||(1+o(1)),\;
k\to\infty.
    \ee
    Inserting \eqref{sof4} into \eqref{l3}, we obtain \eqref{bx5}.
\end{proof}
Now we are in position to prove Theorem \ref{thdj4}.
By Propositions \ref{pr1} and
\ref{pr2} we have
   $$
(E_D(k)-\mathcal{E}_D)^{-1/2}||\tilde{\pi}_{D,\infty}-(\tilde{\zero}\oplus\tilde{\pi}_D(k))|| \leq
    $$
     \bel{bx3}
 C \frac{\|\tilde{\pi}_{D,\infty} \Lambda_k\|}{\|\tilde{\pi}_{D,\infty} \Lambda_k^{1/2}\|} (1 + o(1)) \leq C \|\Lambda_k\|^{1/2} (1 + o(1)), \quad k \to \infty.
    \ee
Now, \eqref{bx3} and \eqref{bx4} imply \eqref{bx1a} and, hence, \eqref{dj34}.\\

(v) {\em End of the proof of Theorem \ref{thdj3}}. By analogy with \eqref{sof1},  define the operator
   $$
R_{D, \infty}(\lambda, A):= {\mathcal F}^* \int^\oplus_{(A,\infty)} (E_D(k)-{\mathcal E}_D + \lambda)^{-1/2}\pi_{D, \infty}(k)\,dk\,{\mathcal F}, \quad \lambda > 0, \quad A \in [-\infty, \infty),
$$
  self-adjoint in $L^2(\rd)$.
The remaining part of the proof of Theorem \ref{thdj3} is based on the following two facts.
First, for any $A \in [-\infty, \infty)$, $\tilde{A} \in (A, \infty)$, and $r \in (0, \infty)$ we have
    \bel{dj40}
    n_*(r; V_*^{1/2} (R_D(\lambda, A) - R_D(\lambda, \tilde{A})) = O_{r, A}(1),
    \ee
    \bel{dj41}
    n_*(r; V_*^{1/2} (R_{D, \infty}(\lambda, A) - R_{D,\infty}(\lambda, \tilde{A})) = O_{r, A}(1)
    \ee
    as $\lambda \downarrow 0$, since the compact operators
    $$
    V_*^{1/2} (R_D(\lambda, A) - R_D(\lambda, \tilde{A})) \quad \mbox{and} \quad V_*^{1/2} (R_{D, \infty}(\lambda, A) - R_{D,\infty}(\lambda, \tilde{A}))
    $$
    admit norm limits as  $\lambda \downarrow 0$.
    Secondly, we have
    $$
\| V_*^{1/2} ( R_{D,\infty}(\lambda, \tilde{A}) -R_D(\lambda;
\tilde{A}))\| \leq
$$
    \bel{lau34}
\|V_*\|_{L^{\infty}(\rd)}^{1/2} \sup_{k > \tilde{A}} (
E_D(k)-{\mathcal E}_D)^{-1/2} \|(\zero \oplus{\pi}_D(k)) -
{\pi}_{D,\infty}\|,
    \ee
for any $\lambda > 0$. By Theorem \ref{thdj4} and \eqref{lau34}, we find that for each $r>0$ there exists $A_0 = A_0(r)$ such that $\tilde{A} > A_0(r)$ implies
$$
\| V_*^{1/2} ( R_{D,\infty}(\lambda, \tilde{A}) - R_D(\lambda ;
\tilde{A}))\| < r,
$$
for any $\lambda > 0$, and, hence,
    \bel{dj42}
    n_*(r; V_*^{1/2} ( R_{D,\infty}(\lambda, \tilde{A}) - R_D(\lambda ;
\tilde{A})) = 0, \quad  \lambda > 0.
    \ee
 Combining \eqref{dj29} with \eqref{dj40}, \eqref{dj41}, and \eqref{dj42}, and applying the Ky Fan inequalities, we get
$$
n_*(1+ \varepsilon; V_*^{1/2}  R_{D,\infty}(\lambda; A)) + O_{\varepsilon, A}(1) \leq
$$
$$
{\mathcal N}_D(\lambda) \leq
$$
    \bel{dj43}
n_*(1 - \varepsilon; V_*^{1/2}  R_{D,\infty}(\lambda; A)) + O_{\varepsilon, A}(1), \quad \lambda \downarrow 0.
    \ee
In order to see that \eqref{dj43} is equivalent to \eqref{dj28}, it suffices to note that the operators $V_*^{1/2} R_{D,\infty}(\lambda; A)$ and $S_{D}(\lambda; A)$
have the same non zero singular values.

\subsection{Proof of Corollary 2.2}
Throughout the subsection we assume that $V$ satisfies \eqref{dj35}.
For $\delta \in (0,1/2)$ introduce the intervals
    $$
    I_- = I_-(\delta) : = (\delta, 1-\delta), \quad I_+ = I_+(\delta) : = (0, 1+\delta).
    $$
For  $m>0$ define  $\Gamma^\pm_\delta(m): L^2(I_\pm) \to L^2(\Omega_\pm)$
    as the operator with integral kernel
    $$
    \pi^{-1/2} m  e^{-bx^2/2} e^{m(x+iy)k} k^{1/2}, \quad k \in I_\pm, \quad (x,y) \in
    \Omega_\pm.
    $$
       {\em Remark}: Introduce the set
          \bel{jul20}
          {\mathcal B}(\Omega_\pm) : = \left\{u \in L^2(\Omega_\pm) \, | \, u \, \mbox{is analytic in} \, \Omega_\pm\right\}
          \ee
          considered as a subspace of the Hilbert space $L^2(\Omega_\pm; e^{-bx^2} dxdy)$. Note that as a functional set ${\mathcal B}(\Omega_\pm)$ coincides with the Bergman space over
          $\Omega_\pm$ (see e.g. \cite[Subsection 3.1]{ha}). Then,
    up to unitary equivalence, the operators $\Gamma_\delta^{\pm}(m)$ map $L^2(I_\pm)$ into ${\mathcal B}(\Omega_\pm)$.\\

    In the following proposition we reduce  the analysis of the asymptotic behaviour of $n_*(r; S_D(\lambda; A))$ to the study of the spectral asymptotics of the operators $\Gamma_\delta^{\pm}(m)$ as $m \to \infty$.

\begin{pr} \label{lauth2}
 Assume that  $V$  satisfies
\eqref{dj35}. Then we have
    \bel{jun5}
    n_*(r; S_D(\lambda; A)) \geq n_*(r (1 + \varepsilon) 2 c_-^{-1/2} \sqrt{b |{\ln{\lambda}}|}; \Gamma^-_\delta (\sqrt{b |{\ln{\lambda}}|})) + O(1),
    \ee
    \bel{jun6}
 n_*(r; S_D(\lambda; A))  \leq n_*(r (1 - \varepsilon) 2 c_+^{-1/2}  ; \Gamma^+_\delta (\sqrt{b |{\ln{\lambda}}|})) + O(1),
    \ee
as $\lambda \downarrow 0$, for all  $A>0$, $\varepsilon
\in (0,1)$, $\delta \in (0,1/2)$ and $r>0$, the constants $c_\pm$ being introduced in \eqref{dj35}.
\end{pr}
In the proof of Proposition \ref{lauth2} we systematically use the following lemma which provides explicitly the first asymptotic term of $E_D(k) - {\mathcal E}_D$ as $k \to \infty$.

\begin{lemma}\label{pr3}
We have
    \bel{asymband}
    E_D(k)-\mathcal{E}_D = \frac{2 b^{1/2}}{\sqrt{\pi}} k e^{-b^{-1}k^2}(1+o(1)),
    \quad k\to \infty.
    \ee
    \end{lemma}
\begin{proof}
From   \eqref{rlambda},  \eqref{sof9a}, and \eqref{dj46}, we have
$$
||\tilde{\pi}_{D,\infty}\Lambda_k\tilde{\pi}_{D,\infty}||=|\langle\rho_k,\tilde{\psi}_{D,\infty}\rangle|^2 = |\langle\rho_k,\varphi\rangle|^2 =
$$
    \bel{dj100}
\left|\alpha_k^{-1/2} \int_{-\infty}^{-k/b}\Psi(x)\varphi(x)dx+
\alpha_k^{1/2} \int_{-k/b}^{\infty}\,\Theta(x)\varphi(x)dx\right|^2.
    \ee
Using \eqref{bx5}, and integrating by parts in \eqref{dj100}, bearing in mind \eqref{asymp}, we obtain \eqref{asymband}.
\end{proof}

Otherwise, the proof of Proposition \ref{lauth2} repeats almost word by word the proof of \cite[Proposition 5.4]{bmr} with $j=1$ so that we omit the details. Yet, we would like to note here that the slight differences in the two proofs are due to the fact that the asymptotics in \eqref{asymband} is of order $k e^{-b^{-1}k^2}$ while the corresponding asymptotics which is used in \cite[Proposition 5.4]{bmr} with $j=1$ is of order $k^{-1} e^{-b^{-1}k^2}$ (see \cite[Proposition 4.2]{bmr} with $j=1$ and $x_0 = 0$). This difference in the asymptotics explains in particular the extra factor $\sqrt{b |\ln{\lambda}|}$ in the argument of the counting function $n_*$ at the r.h.s. of \eqref{jun5}. In order to reduce the analysis in the present article to the analysis in \cite{bmr}, we use at the first stage of the proof of the upper bound \eqref{jun6} the estimate $k^{-1} < k$ for $k \in (A,\infty)$ with $A > 1$, while at the last stage of the proof of the lower bound \eqref{jun5} we use the estimate $k^{-1} > k$ for $k \in (\delta, 1-\delta)$ with $\delta \in (0,1/2)$. \\

In order to complete the proof of Corollary \ref{fdj1}, it suffices to prove the following

\begin{follow} \label{fdj3}
Under the assumptions of Corollary \ref{fdj1} we have
\bel{dj50}
    \lim_{\delta \downarrow 0} \limsup_{\lambda \downarrow 0}
    |\ln{\lambda}|^{-1/2} n_+(r;
    \Gamma_{\delta}^+(\sqrt{b|\ln{\lambda}|})^* \Gamma_{\delta}^+(\sqrt{b|\ln{\lambda}|})) \leq {\mathcal C}_+.
    \ee
\bel{dj51}
    \lim_{\delta \downarrow 0} \liminf_{\lambda \downarrow 0}
    |\ln{\lambda}|^{-1/2} n_+(r \sqrt{b|\ln{\lambda}|};
    \Gamma_{\delta}^-(\sqrt{b|\ln{\lambda}|})^* \Gamma_{\delta}^-(\sqrt{b|\ln{\lambda}|})) \geq {\mathcal C}_-.
    \ee
\end{follow}

\begin{proof}
The proof of \eqref{dj50} can be found in \cite[Subsection 6.2]{bmr}.\\
Let us prove \eqref{dj51} modifying slightly the argument in \cite[Subsection 6.1]{bmr}.
Let $\Omega \subset \rd$ be a bounded domain, and ${\mathcal
    I} \subset (0,\infty)$ be a bounded open non-empty interval.
    For $m>0$  define the operator ${\mathcal G}_{m}(\Omega,
    {\mathcal I}) : L^2({\mathcal I}) \to L^2({\mathcal I})$ as
    the operator with integral kernel
   $$
    \pi^{-1} m^2 \sqrt{k k'} \int_\Omega e^{m(z k + \bar{z} k')}
    d\mu(z), \quad k, k' \in {\mathcal I}.
    $$
    Set
   $\epsilon_- : = \inf_{(x,y) \in \Omega_-} e^{-bx^2}$.
     Then we have
     \bel{jul53}
     \Gamma_{\delta}^-(m)^*
     \Gamma_{\delta}^-(m) \geq \epsilon_- {\mathcal G}_{m}(\Omega_-,
    I_-(\delta)), \quad m>0.
    \ee
    Further, let $\Pi  \subset
    \Omega_-$ be an open non-empty rectangle whose sides are
    parallel to the coordinate axes.
    Assume, without any loss of generality, that $\Pi =
    (\alpha, \beta) \times (-L,L)$ with $0<\alpha<\beta<\infty$
    and $L \in (0,\infty)$. Evidently,
    \bel{jul54}
    {\mathcal G}_{m}(\Omega_-,
    I_-(\delta)) \geq {\mathcal G}_{m}(\Pi,
    I_-(\delta)), \quad m>0.
    \ee
    For $\eta \in \re$ and $\delta \in (0,1/2)$ define the operator
    $G_{\eta, \delta}(m) : L^2(I_-(\delta)) \to L^2(I_-(\delta))$ as the integral operator
    with kernel
    $$
    e^{\eta m(k + k')} \frac{\sin{(mL(k-k'))}}{\pi (k-k')} \frac{2\sqrt{k k'}}{k+k'}, \quad k, k' \in I_-(\delta).
    $$
  Then
    \bel{jul55}
    {\mathcal G}_{m}(\Pi,
    I_-(\delta)) =
     G_{{\beta}, \delta}(m) - G_{{\alpha},
    \delta}(m).
    \ee
 By the mini-max principle, we have
    \bel{cdj2}
    n_+(rm; {G_{{\beta}, \delta}}({ m}) - {G_{{\alpha}, \delta}}(m)) \geq
    n_+(rm e^{-2m {\beta}\delta}; {G_{0, \delta}}({ m}) - {G_{{\alpha} - {\beta}, \delta}}(m))
    \ee
    with $0< \alpha < \beta <\infty$, $m \in (0,\infty)$, $r \in (0,\infty)$. Fix $s \in (0,\infty)$ and find $m_0$ such that $r me^{-2m \beta \delta} < s/2$ for $m > m_0$. Using the fact that the counting function $n_+(r; \cdot)$, $r\in (0,\infty)$, is decreasing, as well as the Weyl inequalities \eqref{lau11}, we get
  $$
   n_+(rm e^{-2m \beta \delta}; {G_{0, \delta}}({ m}) - {G_{{\alpha}-{\beta}, \delta}}(m)) \geq n_+(s/2; {G_{0, \delta}}({ m}) - {G_{{\alpha}-{\beta}, \delta}}(m)) \geq
   $$
    \bel{cdj3}
   n_+(s; G_{0, \delta}({ m})) - n_+(s/2;{G_{{\alpha}-{\beta}, \delta}}(m)).
   \ee
   Evidently, $\|G_{\alpha-\beta, \delta}(m)\| \leq \|G_{{\alpha}-{\beta}, \delta}(m)\|_2$,  and
   $$
   \|G_{{\alpha}-{\beta}, \delta}(m)\|_2^2 = \frac{4}{\pi^2} \int_{I_-(\delta)} \int_{I_-(\delta)} e^{2({\alpha}-{\beta})m(k+k')} \left(\frac{\sin{(mL(k-k'))}}{k-k'}\right)^2 \frac{k k'}{(k+k')^2} dk dk' \leq
   $$
   $$
   \frac{L^2m^2}{\pi^2} e^{4({\alpha}-{\beta})m\delta} (1-2\delta)^2,
   $$
   so that $\lim_{m \to \infty} \|G_{{\alpha}-{\beta}, \delta}(m)\| = 0$. Fix $s \in (0,\infty)$ and find $m_1$ such that $\|G_{{\alpha}-{\beta}, \delta}(m)\| < s/2$ for $m > m_1$. Then
   \bel{cdj5}
   n_+(s/2;{G_{{\alpha}-{\beta}, \delta}}(m)) = 0, \quad m > m_1.
   \ee

Putting together \eqref{jul53} --  \eqref{cdj5},
    we find that  for each $\delta \in (0,1/2)$, $r \in (0,\infty)$, and $s \in (0,\infty)$, we have
    \bel{san7}
    n_+(rm; \Gamma_{\delta}^-(m)^* \Gamma_{\delta}^-(m)) \geq
    n_+(s\epsilon_-^{-1}; G_{0, \delta}(m)),
    \ee
provided that $m > \max{\{m_0,m_1\}}$. Further, let, as above, ${\mathcal I} \subset \re_+$ be an open bounded interval.
Define the
    operator $g_{\mathcal I}(m): L^2(\mathcal I) \to L^2(\mathcal
    I)$, $m>0$,
    as the operator with integral kernel
    $$
\frac{\sin{(mL(k-k'))}}{\pi(k-k')} \frac{2\sqrt{k k'}}{k+k'}, \quad
k,k' \in {\mathcal I}.
    $$

   Note that
   \bel{cdj7}
   G_{0,\delta}(m) = g_{I_-(\delta)}(m).
   \ee
By \cite[Corollary 6.3]{bmr} we have
    \bel{renov5}
    \lim_{m \to \infty} m^{-1} n_+(s; g_{\mathcal I}(m))  =
    \left\{
    \begin{array} {l}
    \frac{L|{\mathcal I}|}{\pi} \quad {\rm if} \quad s \in (0,1), \\
    0 \quad {\rm if} \quad s > 1.
    \end{array}
    \right.
    \ee

    Now fix $s < \epsilon_-$ in \eqref{san7}, take into account \eqref{cdj7}, and put together \eqref{san7}
    and \eqref{renov5}; thus we find that
the
    asymptotic estimate
    $$
\liminf_{\lambda \downarrow 0} |\ln{\lambda}|^{-1/2} n_+(r \sqrt{b |\ln{\lambda}|} ;
    \Gamma_{\delta}^-(\sqrt{b |\ln{\lambda}|})^* \Gamma_{\delta}^-(\sqrt{b |\ln{\lambda}|})) \geq
    \frac{\sqrt{b}L}{\pi}(1-2\delta)
    $$
    holds for every $\delta \in (0,1/2)$. Letting $\delta \downarrow 0$, and optimizing with respect to $L$, we
    obtain \eqref{dj51}.
 \end{proof}
 {\em Remark.} Since
 $$
 n_+(rm ; \Gamma_\delta^-(m)^* \Gamma_\delta^-(m)) \leq n_*(r ; \Gamma_\delta^-(m)^* \Gamma_\delta^-(m))
  $$
  for $r>0$ and $m \geq 1$, our proof of \eqref{dj51} also provides a slightly modified proof
 of the lower bound \cite[Eq. (6.3)]{bmr}. This modification does not use the first inequality in \cite[Eq. (6.10)]{bmr} in whose proof we found a gap; we thank Dr  Marcus Carlsson for drawing our attention to the problem with that inequality.

\section{Proof of the main result for Neumann\\ Hamiltonians} \label{s4}

In this section we prove Theorem \ref{thdj2}. Assume that $0 \leq V \in L_0^{\infty}({\mathcal O})$. Then, similarly to \eqref{dj30},  the Birman - Schwinger principle implies
    \bel{dj60}
    {\mathcal N}_N(\lambda) = n_+(1; V^{1/2} (H_{0,N} - {\mathcal E}_N + \lambda)^{-1} V^{1/2}), \quad \lambda > 0.
    \ee
    Fix $\delta \in (0,\infty)$, and define the orthogonal projection $P_N(\delta) : L^2({\mathcal O}) \to L^2({\mathcal O})$ by
    $$
    P_N(\delta) : = {\mathcal F}^* \int^\oplus_{(k_*-\delta,k_*+\delta)}  \pi_N(k) \,dk {\mathcal F},
    $$
    where $\pi_N(k) : = |\psi_N(\cdot; k)\rangle \langle \psi_N(\cdot; k)|$, the function $\psi_N(\cdot; k)$ being defined at the end of Subsection \ref{ss21}, and $k_*$ is the  point where $E_N$ attains its minimum ${\mathcal E}_N$.
    Then by analogy with \eqref{dj31}, the estimates
    $$
    n_+(1; V^{1/2} (H_{0,N} - {\mathcal E}_N + \lambda)^{-1} P_N(\delta) V^{1/2}) \leq n_+(1; V^{1/2} (H_{0,N} - {\mathcal E}_N + \lambda)^{-1} V^{1/2}) \leq
    $$
    \bel{dj61}
    n_+(1 - \varepsilon; V^{1/2} (H_{0,N} - {\mathcal E}_N + \lambda)^{-1} P_N(\delta) V^{1/2}) +
    n_+(\varepsilon; V^{1/2} (H_{0,N} - {\mathcal E}_N + \lambda)^{-1} (I-P_N(\delta)) V^{1/2})
    \ee
    hold true for any $\lambda > 0$ and $\varepsilon \in (0,1)$. Next, similarly to \eqref{dj31a}, we have
    \bel{dj62}
    n_+(\varepsilon; V^{1/2} (H_{0,N} - {\mathcal E}_N + \lambda)^{-1} (I-P_N(\delta)) V^{1/2}) = O_{\varepsilon, \delta}(1), \quad \lambda \downarrow 0.
    \ee
Define  $ {\mathcal R}_N(\lambda ; \delta) : L^2(-\delta, \delta) \to  L^2(\mathcal{O})$ as the operator with integral kernel
$$
(2\pi)^{-1/2} V(x,y)^{1/2} e^{iky} \psi_N(x; k_* + k) (E_N(k_* + k) - {\mathcal E}_N + \lambda)^{-1/2}, \quad k \in (-\delta, \delta), \quad (x,y) \in {\mathcal{O}}.
$$
It is easy to see that $V^{1/2} (H_{0,N} - {\mathcal E}_N + \lambda)^{-1} P_N(\delta) V^{1/2}$ and ${\mathcal R}_N(\lambda ; \delta)^*  {\mathcal R}_N(\lambda ; \delta)$ have the same non zero eigenvalues. Therefore, for each $r > 0$ and $\lambda > 0$ we have
    \bel{dj63}
    n_+(r^2; V^{1/2} (H_{0,N} - {\mathcal E}_N + \lambda)^{-1} P_N(\delta) V^{1/2}) = n_*(r;  {\mathcal R}_N(\lambda ;\delta)).
    \ee
    Next, define  $R_N(\lambda ; \delta) : L^2(-\delta, \delta) \to  L^2(\mathcal{O})$ as the operator with integral kernel
$$
(2\pi)^{-1/2} V(x,y)^{1/2} e^{iky} \psi_N(x; k_* ) (\mu k^2 + \lambda)^{-1/2}, \quad k \in (-\delta, \delta), \quad (x,y) \in {\mathcal{O}},
$$
the number $\mu$ being introduced in \eqref{dj75}. Applying the Ky Fan inequalities, we get
$$
n_*(r (1 + \varepsilon); R_N(\lambda ;\delta)) -  n_*(r\varepsilon;  {\mathcal R}_N(\lambda ;\delta) - R_N(\lambda ;\delta)) \leq
$$
$$
n_*(r;  {\mathcal R}_N(\lambda ;\delta)) \leq
$$
    \bel{dj64}
    n_*(r (1 - \varepsilon); R_N(\lambda ;\delta)) +  n_*(r\varepsilon;  {\mathcal R}_N(\lambda ;\delta) - R_N(\lambda ;\delta)), \quad \lambda \downarrow 0,
    \ee
    for any $r>0, \lambda > 0, \varepsilon \in (0,1)$. Our next goal is to show that
     \bel{dj65}
     n_*(r;  {\mathcal R}_N(\lambda ;\delta) - R_N(\lambda ;\delta)) = O_{r, \delta}(1), \quad \lambda \downarrow 0.
     \ee
     To this end, fix $\rho \in (0,\infty)$, set
     $$
     B_\rho : = \left\{(x,y) \in {\mathcal O} \, | \, x^2 + y^2 < \rho^2\right\},
     $$
     $$
     U_{1,\rho} : = \one_{B_\rho} V^{1/2}, \quad  U_{2,\rho} : = \one_{\rd \setminus {B_\rho}} V^{1/2},
     $$
     and write
     $$
     {\mathcal R}_N(\lambda ;\delta) - R_N(\lambda ;\delta) = \left( U_{1,\rho} +  U_{2,\rho}\right)(T_1 + T_2)
     $$
     where
     $T_1 = T_1(\lambda; \delta) : L^2(-\delta, \delta) \to L^2({\mathcal O})$ is the operator with integral kernel
     $$
     (2\pi)^{-1/2} e^{iky} \psi_N(x; k_* + k) \left((E_N(k_* + k) - {\mathcal E}_N + \lambda)^{-1/2} - (\mu k^2 + \lambda)^{-1/2}\right),
     $$
     and
     $T_2 = T_2(\lambda; \delta) : L^2(-\delta, \delta) \to L^2({\mathcal O})$ is the operator with integral kernel
     $$
     (2\pi)^{-1/2}  e^{iky} \left(\psi_N(x; k_* + k) - \psi_N(x; k_*)\right) (\mu k^2 + \lambda)^{-1/2},
     $$
     with $k \in (-\delta, \delta)$, and $(x,y) \in   {\mathcal O}$. It is not difficult to see that for any $\lambda > 0$ we have
     \bel{dj66}
     \|U_{1,\rho} T_j(\lambda; \delta)\|_2^2 \leq \frac{\delta}{\pi} \sup_{x \in \re_+} \int_\re U_{1,\rho}(x,y)^2 dy  \sup_{k \in (-\delta, \delta)} Q_j(k)^2,
     \quad j=1,2, \ee
     \bel{dj67}
     \|U_{2,\rho} T_j(\lambda;  \delta)\| \leq \|U_{2,\rho}\|_{L^{\infty}(\mathcal{O})} \sup_{k \in (-\delta, \delta)} |Q_j(k)|, \quad j=1,2,
     \ee
     where
     $$
     Q_1(k) : = \frac{E_N(k_* + k) - {\mathcal E}_N - \mu k^2}{\sqrt{\mu} |k| \sqrt{E_N(k_* + k) - {\mathcal E}_N}\left(\sqrt{\mu} |k| + \sqrt{E_N(k_* + k) - {\mathcal E}_N}\right)},
     $$
     $$
      Q_2(k)^2 : =  \mu^{-1} \int_{\re_+}\left(\frac{\partial \psi_N}{\partial k}(x; k_* + k)\right)^2 dx.
    $$
    Now the Ky Fan inequalities imply
    \bel{dj68}
    n_*(r;  {\mathcal R}_N(\lambda ;\delta) - R_N(\lambda ;\delta)) \leq \sum_{j=1,2} n_*(r/2; U_{j,\rho} (T_1 + T_2)).
    \ee
    Since $V \in L_0^{\infty}({\mathcal O})$,  we can choose, using \eqref{dj67}, the number $\rho = \rho(r)$ so large that $\|U_{2,\rho} (T_1 + T_2)\| \leq r/2$, and hence
        \bel{dj69}
    n_*(r/2; U_{2,\rho} (T_1 + T_2)) = 0.
    \ee
    On the other hand, by \eqref{dj66} and \eqref{dj36} with $p=2$ we have
     \bel{dj70}
    n_*(r/2; U_{1,\rho} (T_1 + T_2)) \leq 4r^{-2} \|U_{1,\rho} (T_1 + T_2)\|_2^2 = O(1), \quad \lambda \downarrow 0.
    \ee
    Putting together \eqref{dj68} -- \eqref{dj70}, we obtain \eqref{dj65}. \\
    Further, let ${\mathcal S}_N(\lambda; \delta) : L^2(\re) \to L^2(\re)$ be the operator with integral kernel
    $$
    (2\pi)^{-1/2} v(y)^{1/2} e^{iky} (\mu k^2 + \lambda)^{-1/2} \one_{(-\delta, \delta)}(k), \quad k \in \re, \quad y \in \re,
    $$
    the potential $v$ being defined in \eqref{dj12}. It is easy to see that the operators $R_N(\lambda;\delta)$ and ${\mathcal S}_N(\lambda; \delta)$ have the same non zero singular values, and hence
    \bel{dj74}
     n_*(r;  R_N(\lambda ;\delta)) =  n_*(r; {\mathcal S}_N(\lambda ;\delta)), \quad \lambda > 0, \quad r> 0.
     \ee
     Finally, let $S_N(\lambda) : L^2(\re) \to L^2(\re)$ be the operator with integral kernel
    $$
    (2\pi)^{-1/2} v(y)^{1/2} e^{iky} (\mu k^2 + \lambda)^{-1/2}, \quad k \in \re, \quad y \in \re.
    $$
    By the Ky Fan inequalities,
    $$
n_*(r (1 + \varepsilon); S_N(\lambda)) -  n_*(r\varepsilon; {\mathcal S}_N(\lambda ;\delta) - S_N(\lambda)) \leq
$$
$$
n_*(r; {\mathcal S}_N(\lambda ;\delta)) \leq
$$
    \bel{dj71}
    n_*(r (1 - \varepsilon); S_N(\lambda)) +  n_*(r\varepsilon; {\mathcal S}_N(\lambda ;\delta) - S_N(\lambda)), \quad \lambda \downarrow 0,
    \ee
    for any $r>0, \lambda > 0, \varepsilon \in (0,1)$. Arguing as in the derivation of \eqref{dj65}, we easily obtain
     \bel{dj72}
     n_*(r; {\mathcal S}_N(\lambda ;\delta) - S_N(\lambda)) = O_{r, \delta}(1), \quad \lambda \downarrow 0.
     \ee
     By the Birman -- Schwinger principle,
     \bel{dj73}
     n_*(r; S_N(\lambda)) =
     {\rm Tr}\,\one_{(-\infty, -\lambda)}\left( - \mu \frac{d^2}{dy^2} - r^{-2} v\right).
       \ee
       Combining \eqref{dj60} -- \eqref{dj65} with \eqref{dj74} -- \eqref{dj73}, we obtain \eqref{dj26}. \\

       {\bf Acknowledgements.} The authors are grateful for hospitality and financial
support to the  Mittag-Leffler Institute, Sweden, where most of this work was done within the the framework of the Programme ``{\em Hamiltonians in Magnetic Fields}", Fall 2012.\\
V. Bruneau was partially supported by ANR project NOSEVOL (ANR 2011
BS0101901).
P. Miranda was partially supported by the Chilean Science Foundation
{\em Fondecyt} under Grant 3120087.
 G. Raikov  was partially supported by {\em N\'ucleo Cient\'ifico ICM} P07-027-F ``{\em
Mathematical Theory of Quantum and Classical Magnetic Systems"} within the framework of the {\em International
Spectral Network}, as well as
 by  {\em Fondecyt} under Grant 1090467.\\

{\sc Vincent Bruneau }\\
Universit\'e Bordeaux I, Institut de Math\'ematiques de
Bordeaux,\\
UMR CNRS 5251,  351, Cours de la Lib\'eration, 33405 Talence,
France\\
E-Mail: vbruneau@math.u-bordeaux1.fr\\

{\sc Pablo Miranda}\\
Departamento de F\'isica,
Facultad de F\'isica,\\
Pontificia Universidad Cat\'olica de Chile,
Vicu\~na Mackenna 4860, Santiago de Chile\\
E-Mail: pmiranda@fis.puc.cl\\

{\sc Georgi Raikov}\\
 Departamento de Matem\'aticas, Facultad de
Matem\'aticas,\\ Pontificia Universidad Cat\'olica de Chile,
Vicu\~na Mackenna 4860, Santiago de Chile\\
E-Mail: graikov@mat.puc.cl

\end{document}